\documentclass[12 pt]{article}
\usepackage{geometry}
\usepackage{amssymb}
\usepackage{amsmath, amsthm, amsbsy, centernot, mathrsfs}
\usepackage{tikz-cd}
\usepackage{amssymb}
\usepackage[all]{xy}
\usepackage{enumerate}
\usepackage{amsfonts}
\usepackage{mathtools}

\usepackage{sectsty}
\sectionfont{\normalsize\bfseries}
\subsectionfont{\normalsize\bfseries}

\newtheorem{theorem}{Theorem}
\numberwithin{theorem}{section}
\newtheorem{proposition}[theorem]{Proposition}
\newtheorem{lemma}[theorem]{Lemma}
\newtheorem{corollary}[theorem]{Corollary}
\theoremstyle{definition}
\newtheorem{definition}[theorem]{Definition}
\newtheorem{example}[theorem]{Example}

\theoremstyle{remark}

\begin{document}
\def\dim{{\rm{dim}}}
\def\rk{{\rm{rk}}}
\def\Hom{{\rm{Hom}}}
\def\End{{\rm{End}}}
\def\Aut{{\rm{Aut}}}
\def\Ext{{\rm{Ext}}}
\def\Res{{\rm{Res}}}
\def\Ann{{\rm{Ann}}}
\def\Z{\mathbb{Z}}
\def\z{\mathcal{Z}}
\def\N{\mathbb{N}}
\def\R{\mathbb{R}}
\def\O{\mathcal{O}}
\def\C{\mathscr{C}}
\def\ker{{\rm{Ker}}}
\def\im{{\rm{Im}}}
\def\char{{\rm{char}}}
\def\GL{{\rm{GL}}}
\def\F{\mathbb{F}}
\def\t{\tau}
\def\l{\lambda}
\def\L{\mathscr{L}}
\def\ind{{\rm{IND}}}
\def\sgn{{\rm{SGN}}}
\def\Ind{{\rm{Ind}}}
\def\Res{{\rm{Res}}}
\def\mod{{\rm{mod}}}
\def\S{\mathfrak{S}}
\def\soc{{\rm{soc}}}
\def\head{{\rm{head}}}
\def\rad{{\rm{rad}}}
\def\b{\mathfrak{b}}
\def\e{\mathfrak{e}}
\def\p{\mathfrak{p}}
\def\u{\mathfrak{u}}
\def\a{\alpha}
\def\St{{\rm{St}}}
\def\H{\mathscr{H}}
\def\F{\mathfrak{F}}
\def\G{\Gamma}
\def\D{\Delta}
\def\d{\delta}
\def\r{\rho}
\def\SL{{\rm{SL}}}
\def\Rad{{\rm{Rad}}}
\def\Cor{{\rm{Cor}}}

\title{Bounds on the Dimension of $\Ext$ for Finite Groups of Lie Type}

\author{Veronica Shalotenko \\ \footnotesize{vvs9cc@virginia.edu} \\ \footnotesize{Department of Mathematics, University of Virginia, Charlottesville, VA 22903}}

\maketitle

\begin{abstract}
{\footnotesize{Let $G$ be a finite group of Lie type defined in characteristic $p$, and let $k$ be an algebraically closed field of characteristic $r>0$. We will assume that $r \neq p$ (so, we are in the non-defining characteristic case). Let $V$ be a finite-dimensional irreducible left $kG$-module. In 2011, Guralnick and Tiep found bounds on the dimension of $H^1(G,V)$ in non-defining characteristic, which are independent of $V$. The aim of this paper is to generalize the work of Gurlanick and Tiep. We assume that $G$ is split and use methods of modular Harish-Chandra theory to find bounds on the dimension of $\Ext^1_{kG}(Y,V)$, where $Y$ and $V$ are irreducible $kG$-modules. We then use Dipper and Du's algorithms to illustrate our bounds in a series of examples.}}

\end{abstract}

\section{Introduction}

Let $q$ be a power of a prime $p$, let $G$ be a finite group of Lie type over the finite field of $q$ elements $\mathbb{F}_q$, and let $k$ be an algebraically closed field of characteristic $r>0$, $r \neq p$.\footnote{Finite groups of Lie type are discussed in detail in \cite{carter}.} We will work with left $kG$-modules, and all $kG$-modules will be assumed to be finite-dimensional over $k$. When $G$ is a finite group of Lie type, the $BN$-pair of $G$ is split in the sense of \cite[2.5]{carter}, i.e. this $BN$-pair satisfies the additional axioms (a) $B=UT$, where $T=B \cap N$ and $U$ is the largest normal $p$-subgroup of $B$, and (b) $\cap_{n \in N}~ nBn^{-1}=T$.\footnote{We will follow \cite{gt} and refer to $B$ as a Borel subgroup, $T$ a maximal torus, and $U$ the unipotent radical. $B$ is not a Borel subgroup in the sense of algebraic groups any more than $T$ is a maximal torus or $U$ is a unipotent radical. However, this terminology proves to be very convenient when discussing the cohomology of finite groups of Lie type.} Thus, we can take advantage of the results of modular Harish-Chandra theory outlined in \cite{gj}. \\

In this paper, we combine techniques of modular Harish-Chandra theory with Guralnick and Tiep's methods for studying $1$-cohomology in non-defining characteristic in order to find bounds on the dimension of $\Ext^1_{kG}(Y,V)$ when $Y$ and $V$ are certain irreducible $kG$-modules. In Section 4, we show that $\Ext^1_{kG}(Y,V)=0$ when $Y$ and $V$ are principal series representations belonging to distinct principal series. In Section 5, we show that $\dim~ \Ext^1_{kG}(Y,V) \leq [W:W(T,X)]~|W|+{\rm{min}}(\dim~Y,\dim~V)e$ (where $e$ is the $r$-rank of of $T$) when $Y$ and $V$ both belong to a principal series ${\rm{Irr}}_k(G|(T,X))$ for some one-dimensional $kT$-module $X$.\footnote{$W(T,X)$ is the inertia group of $X$ (see Definition \ref{inertia}).} In Section 7, we assume that $Y$ is a unipotent principal series representation and that $V$ is an irreducible $kG$-module which lies outside the unipotent principal series. Under certain additional assumptions on the group $G$, we show that there exists a parabolic subgroup $W_J$ of $W$ (which depends only on $V$) such that  $\dim~\Ext^1_{kG}(Y,V) \leq [W:W_J]$. In Section 7, we give examples of this bound in the case that $G$ is a finite general linear group. These results were originally proved in the author's thesis \cite{thesis}.

\section{Motivation}

There are three distinct cases to consider in the representation theory of $kG$ (where $G$ is a finite group of Lie type over $\mathbb{F}_q$, with $q$ a power of a prime $p$): $r=0$ (the characteristic 0 case), $r=p$ (the defining characteristic case), and $r>0$, $r \neq p$ (the non-defining, or cross-characteristic case). One goal of research in the defining and non-defining characteristic cases is to compute the dimensions of Ext groups $\Ext^i(G,V)$, where $V$ is an irreducible $kG$-module and $i \geq 1$. In the defining characteristic, it is known that such bounds exist when the rank $i$ is fixed (this is due to Cline, Parshall, and Scott \cite{cps09} and Parshall and Scott \cite{ps11} in the $i=1$ case and to Bendel, Nakano, Parshall, Pillen, Scott, and Stewart \cite{bnppss} in the $i>1$ case). \\

In 2011, Guralnick and Tiep \cite{gt} published the following bounds on 1-cohomology in non-defining characteristic. \\

\begin{theorem} (\cite[Cor. 3.3, 6.5]{gt})
Let $G$ be a finite group of Lie type defined in characteristic $p$, and let $\char(k)=r>0$, with $r \neq p$. Let $W$ be the Weyl group of $G$, and let $e$ be the Lie rank of $G$. If $V$ is an irreducible $kG$-module, then 
$$
\dim~H^1(G,V) \leq
\begin{cases}
1 &\text{ if } V^B=0 \\
|W|+e &\text{ if } V^B \neq 0
\end{cases}
$$
(where $B$ is a Borel subgroup of $G$). \\
\end{theorem}

An irreducible $kG$-module $V$ satisfies the condition $V^B \neq 0$ if and only if $V$ belongs to the unipotent principal Harish-Chandra series $\text{Irr}_k(G|B)$ (see Section 3.4). Hence, \cite[Cor. 3.3, 6.5]{gt} may be restated as follows: given an irreducible $kG$-module $V$,
$$\dim~ \Ext^1_{kG}(k,V) \leq
\begin{cases}
1 &\text{ if  } V \not\in \text{Irr}_k(G|B) \\
|W|+e &\text{ if } V \in \text{Irr}_k(G|B). \\
\end{cases} 
$$

The work presented in this paper stems from the observation that Guralnick and Tiep's bounds can be interpreted in terms of Harish-Chandra theory. In Sections 4-7, we use techniques of modular Harish-Chandra theory to generalize \cite[Cor. 3.3, 6.5]{gt} and find bounds on the dimension of $\Ext^1$ between irreducible $kG$-modules.

\section{Modular Harish-Chandra Theory}

Our summary of Harish-Chandra theory is based on \cite[Sec. 4.2]{gj}. As above, let $G$ be a finite group of Lie type defined in characteristic $p$, and let $k$ be an algebraically closed field of characteristic $r>0$, $r \neq p$ (for the remainder of this paper, we will restrict our attention to the non-defining characteristic case). Let $(W,S)$ be the Coxeter system corresponding to the $BN$-pair structure of $G$.

\subsection{Harish-Chandra Induction and Restriction}

Given a subset $J \subseteq S$, let $P_J$ be the standard parabolic subgroup of $G$ corresponding to $J$. Let $U_{P_J}$ be the largest normal $p$-subgroup of $P_J$, and let $L_J$ be a Levi subgroup of $P_J$. In the notation of \cite[4.2.1]{gj}, let $\mathscr{P}_G=\{ ^{n}{P_J}~ |~ J \subseteq S, ~n \in N \}$ and $\L_G= \{^{n}{L_J} ~ |~ J \subseteq S, ~n \in N \}$.\footnote{For any any subgroup $H \leq G$, $^{n}H=nHn^{-1}$.} \\

Let $n \in N$ and let $X$ be a (left) $kL$-module. We can define a $k({^{n}{L}})$-module structure on $X$ by setting $nln^{-1}.x=l.x$ for any $l \in L$ and $x \in X$. The resulting $k({^{n}{L}})$-module will be denoted by ${^{n}{X}}$. Let $w \in W=N/T$ and let $n \in N$ be a representative of $w$. In this case, we define ${^{w}{L}} :={^{n}{L}}$ and ${^{w}{X}} :={^{n}{X}}$ (${^{n}{L}}$ and ${^{n}{X}}$ are well-defined since ${^{nt}{L}}={^{n}{L}}$ and ${^{n}{X}} \cong {^{nt}{X}}$ for any $n \in N \leq N_G(T)$ and $t \in T$). \\

Let $P \in \mathscr{P}_G$ and let $L \in \L_G$ be a Levi complement in $P$; in this case, $P=U_P \rtimes L$. Let $kL-\mod$ denote the category of (finite dimensional) left $kL$-modules, and let $kG-\mod$ denote the category of (finite dimensional) left $kG$-modules. There is a Harish-Chandra induction functor
\begin{equation}\label{18harishinduction} 
R_{L \subseteq P}^G: kL-\text{mod} \to kG-\text{mod},
\end{equation}
defined by $R_{L \subseteq P}^G(X)=\Ind_P^G( \tilde{X})$ for all $X \in kL-\mod$, where $\tilde{X}$ denotes the inflation of $X$ from $L$ to $P$ via the surjective homomorphism $P \to L$ with kernel $U_P$. \\

There is also a Harish-Chandra restriction functor
\begin{equation}\label{18harishrestriction}
^*{R}^G_{L \subseteq P}: kG-\text{mod} \to kL-\text{mod}.
\end{equation}
Given a $kG$-module $Y$, $^*{R}^G_{L \subseteq P} (Y)=Y^{U_P}$ (which has the structure of a $kL$-module since $U_P$ is a normal subgroup of $P$). \\

A key feature of Harish-Chandra induction is the following independence property, which was proved by Howlett and Lehrer \cite{hl} and Dipper and Du \cite{ddother}. Let $L, M \in \L_G$, and suppose that $L$ is a Levi complement of $P \in \mathscr{P}_G$ and $M$ is a Levi complement of $Q \in \mathscr{P}_G$. Let $X \in kL-\mod$ and $X' \in kM-\mod$. If $M={^{n}{L}}$ and $X' \cong {^n{X}}$ for some $n \in N$, then $R_{L \subseteq P}^G(X) \cong R_{M \subseteq Q}^G(X')$. As a particular application of this independence property, we have that the Harish-Chandra induction functor $R_{L \subseteq P}^G$ is independent of the choice of parabolic subgroup $P \in \mathscr{P}_G$ containing $L$. Similarly, the Harish-Chandra restriction functor $^*{R}^G_{L \subseteq P}$ is independent of the parabolic subgroup $P$ containing $L$. Therefore, we will omit the parabolic subgroup $P$ and write $R_L^G$ and $^*{R}^G_L$ for the Harish-Chandra induction and restriction functors.

\subsection{Some Properties of Harish-Chandra Induction and Restriction}

\paragraph{Adjointness.}

For any Levi subgroup $L \in \L_G$, $R_L^G$ and $^*{R}^G_L$ are exact. The functors $R_L^G$ and  $^*{R}^G_L$ are each other's two-sided adjoints.

\paragraph{Transitivity.}

Harish-Chandra induction and restriction are transitive. Suppose that $L, M \in \L_G$ are such that $L \subseteq M$. If $X \in kL-\mod$, then $R_L^G(X) \cong R_M^G(R_L^M(X)).$ If $Y \in kG-\mod$, then $^*{R}^G_L(Y) \cong {^*{R}}^M_L(^*{R}^G_M(Y)).$

\paragraph{Mackey decomposition.}

As in the case of ordinary induction and restriction, we have a Mackey decomposition formula for Harish-Chandra induction and restriction. Suppose $L, M \in \L_G$ are Levi complements of the parabolic subgroups $P, Q \in \mathscr{P}_G$, respectively. Let $X$ be a $kL$-module, and let $D(Q,P)$ denote a full set of $(Q,P)$-double coset representatives in $G$.  The Mackey formula provides the following direct sum decomposition of the $kM$-module $^{*}{R}^G_M(R_L^G(X))$:
$$
^{*}{R}^G_M(R_L^G(X)) \cong \underset{n \in D(Q,P)}{\bigoplus} R^M_{{^{n}{L}} \cap M} ({^{*}{R}}^{^{n}{L}}_{{^{n}{L}} \cap M}({^{n}{X}})).
$$

\paragraph{Harish-Chandra induction and the linear dual.}

Let $L \in \L_G$ be the Levi complement of a parabolic subgroup $P \in \mathscr{P}_G$, and suppose that $X$ is a left $kL$-module. Let $X^*$ be the $k$-linear dual of $X$, viewed as a right $kL$-module. Since ordinary induction commutes with taking duals in the case of finite groups, we have $(R_L^G(X))^* \cong R_L^G(X^*)$.

\subsection{Cuspidal Modules and Harish-Chandra Series.} 

A $kG$-module $Y$ is called cuspidal if $^*{R}^G_L(Y)=0$ for all $L \in \L_G$ such that $L \subsetneq G$. (This definition extends to $kL$-modules for any $L \in \L_G$; a $kL$-module $X$ is cuspidal if $^*{R}^L_{L'}(X)=0$ for all $L' \in \L_G$ such that $L' \subsetneq L$.) \\

Let ${\rm{Irr}}_k(G)$ denote a full set of non-isomorphic irreducible $kG$-modules. Given a pair $(L,X)$ with $L \in \L_G$ and $X$ an irreducible cuspidal $kL$-module, let ${\rm{Irr}}_k(G|(L,X))$ be the subset of ${\rm{Irr}}_k(G)$ consisting of all $Y \in {\rm{Irr}}_k(G)$ such that $L \in \L_G$ is minimal with $^*{R}^G_L(Y) \neq 0$ and $X$ is a composition factor of $^*{R}^G_L(Y)$. The set ${\rm{Irr}}_k(G|(L,X))$ is the Harish-Chandra series corresponding to the pair $(L,X)$. If $Y \in {\rm{Irr}}_k(G|(L,X))$, we will say that $L$ is a Harish-Chandra vertex of $Y$ and that $X$ is a Harish-Chandra source of $Y$ (this terminology is used in \cite{ddother} and \cite{dd}). \\

We summarize several properties of Harish-Chandra series, which are stated in \cite[Sec. 4.2]{gj} and were originally proved by Hiss \cite{hiss3}. \\

\begin{proposition} (Properties of Harish-Chandra series) 
\begin{enumerate}[(a)]
\item
Let $L,L' \in \L_G$, let $X$ be a cuspidal irreducible $kL$-module, and let $X'$ be a cuspidal irreducible $kL'$-module. Then, ${\rm{Irr}}_k(G|(L,X))={\rm{Irr}}_k(G|(L',X'))$ if and only if there exists some $n \in N$ with $L'={^n{L}}$ and $X' \cong {^n{X}}$.
\item
The Harish-Chandra series ${\rm{Irr}}_k(G|(L,X))$ consists of the irreducible $kG$-modules which occur in the head (or, equivalently, the socle) of $R_L^G(X)$. 
\item
The set ${\rm{Irr}}_k(G)$ is partitioned by the distinct Harish-Chandra series.
\end{enumerate}
\end{proposition}

\subsection{The Principal Series Representations}

Since every $kT$-module is cuspidal, there is a Harish-Chandra series of the form ${\rm{Irr}}_k(G | (T,X))$ for every irreducible $kT$-module $X$. The irreducible representations of $G$ belonging to a Harish-Chandra series of the form ${\rm{Irr}}_k(G | (T,X))$ are called {\textit{principal series representations}}. (Since $T$ is abelian, an irreducible $kT$-module $X$ must be one-dimensional.) \\

The principal Harish-Chandra series corresponding to the pair $(T,k)$ (where $k$ is viewed as the trivial irreducible $kT$-module) is called the unipotent principal series and is denoted by ${\rm{Irr}}_k(G|B)$. Since $R_T^G(k) \cong k|_B^G$, ${\rm{Irr}}_k(G|B)$ consists of the irreducible $kG$-modules which can be found in both the head and socle of the permutation module $k|_B^G$. If $V \in {\rm{Irr}}_k(G)$, then the multiplicity of $V$ in the head of $k|_B^G$ is $[k|_B^G:V] = \dim~\Hom_{kG}(k|_B^G,V)=\dim~\Hom_{kB}(k,V) = \dim~V^B$. Similarly, $[\soc(k|_B^G):V]=\dim~V^B$. Thus, an irreducible $kG$-module $V$ belongs to the unipotent principal series ${\rm{Irr}}_k(G|B)$ if and only if $V^B \neq 0$.

\subsection{Hecke Algebra Associated with a Harish Chandra Series}

Let $L \in \L_G$ and let $X$ be a cuspidal irreducible (left) $kL$-module. Let $R_L^G$ be the Harish-Chandra induction functor from the category of left $kL$-modules to the category of left $kG$-modules. We define 
$$\H(L,X) := \End_{kG}(R_L^G(X))^{\text{op}},$$
where $\End_{kG}(R_L^G(X))^{\text{op}}$ denotes the opposite of the endomorphism algebra $\End_{kG}(R_L^G(X))$. The algebra $\H(L,X)$ is called the Hecke algebra associated with the pair $(L,X)$. \\

The category of left $kG$-modules is related to the category of left $\H(L,X)$-modules via the Hom functor 
$$\F_{R_L^G(X)}: kG-\mod \to \H(L,X)-\mod, \text{ } Y \mapsto \Hom_{kG}(R_L^G(X),Y).$$

\begin{definition}\label{inertia}
Given a pair $(L,X)$, where $L \in \L_G$ is a Levi subgroup of $G$ and $X$ is a cuspidal irreducible $kL$-module, the inertia group of $X$ is $\mathscr{W}(L,X) := \{ n \in (N_G(L) \cap N)L ~|~ {^{n}{X}} \cong X \}/L$ \cite[4.2]{gj}.  
\end{definition}

The Hecke algebra $\H(L,X)$ has a basis parameterized by $\mathscr{W}(L,X)$ \cite[Thm. 4.2.12]{gj} (the case of $\char~k = 0$ was originally handled in \cite{hl1} and the case of $\char~k>0$ was originally handled in \cite{ghm}).

\section{A Bound on the Dimension of $\Ext^1$ Between Irreducible Modules in Distinct Principal Series}

Let $G$ be a finite group of Lie type defined in characteristic $p$, and let $k$ be an algebraically closed field of characteristic $r>0$, $r \neq p$. In this section, we will assume that the irreducible $kG$-module $Y$ belongs to the principal series ${\rm{Irr}}_k(G | (T,X))$ and that the irreducible $kG$-module $V$ belongs to the principal series ${\rm{Irr}}_k(G | (T,X'))$, with ${\rm{Irr}}_k(G | (T,X)) \neq {\rm{Irr}}_k(G | (T,X'))$. Our goal is to prove that $\Ext^1_{kG}(Y,V)=0$. Since $Y \in {\rm{Irr}}_k(G | (T,X))$, $Y$ is in the head of $R_T^G(X)$. Using a notation analogous to \cite{gt}, let $\L^0$ denote a maximal submodule of $R_T^G(X)$ with $R_T^G(X)/\L^0 \cong Y$. In the next result, we generalize the proof of \cite[Theorem 2.2]{gt} to show that for any irreducible $kG$-module $Z$ with $Z \centernot\in {\rm{Irr}}_k(G | (T,X))$, $\dim~\Ext^1_{kG}(Y,Z)$ is determined by the composition multiplicity of $Z$ in $\head(\L^0)$. \\

\begin{theorem}\label{headcount}
Let $Y$ be an irreducible $kG$-module in the principal series ${\rm{Irr}}_k(G | (T,X))$ (where $X$ is a one-dimensional $kT$-module). If $Z$ is an irreducible $kG$-module such that $Z \centernot\in {\rm{Irr}}_k(G | (T,X))$, then $\dim~\Ext^1_{kG}(Y,Z) = [\head(\L^0):Z]$ (where $ [\head(\L^0):Z]$ denotes the multiplicity of $Z$ as a composition factor of $\head(\L^0)$).
\end{theorem}

\begin{proof}
By definition, $\L^0$ fits into a short exact sequence of $kG$-modules $0 \to \L^0 \to R_T^G(X) \to Y \to 0$. Since $Z \centernot\in {\rm{Irr}}_k(G | (T,X))$, $\Hom_{kG}(Y,Z)=0$ and $\Hom_{kG}(R_T^G(X),Z)=0$. Thus, the long exact sequence in $\Ext$ yields
$$0 \to  \Hom_{kG}(\L^0,Z) \to \Ext^1_{kG}(Y,Z) \to \Ext^1_{kG}(R_T^G(X),Z) \to \cdots.$$
To prove the statement of the theorem, it is enough to show that $\Ext^1_{kG}(R_T^G(X),Z)=0$. For, if $\Ext^1_{kG}(R_T^G(X),Z)=0$, then  $\Ext^1_{kG}(Y,Z) \cong  \Hom_{kG}(\L^0,Z)$ by the exactness of the sequence above, which means that $\dim~\Ext^1_{kG}(Y,Z)=\dim~\Hom_{kG}(\L^0,Z)=[\head(\L^0):Z]$. \\

We will now prove that $\Ext^1_{kG}(R_T^G(X),Z)=0$. By definition of $R_T^G(X)$ and by the Eckmann-Shapiro Lemma, we have $\Ext^1_{kG}(R_T^G(X),Z) = \Ext^1_{kG}(\tilde{X}|_B^G,Z) \cong \Ext^1_{kB}(\tilde{X},Z)$ (where $\tilde{X}$ is the inflation of $X$ from $T$ to $B$ via the surjective homomorphism $B \twoheadrightarrow T$ with kernel $U$). Thus, it suffices to prove that $\Ext^1_{kB}(\tilde{X},Z)=0$. \\

Using the notation of \cite{gt}, let $A$ be the biggest normal subgroup of $B$ of order prime to $r=\char(k)$. The quotient group $B/A$ is an $r$-group. We claim that $\Hom_{kA}(\tilde{X},Z)=0$. Since $Z \not\in {\rm{Irr}}_k(G|(T,X))$, $0 = \Hom_{kG}(R_T^G(X), Z) = \Hom_{kG}( \tilde{X}|_B^G, Z) \cong \Hom_{kB}(\tilde{X},Z)$ (where the isomorphism $\Hom_{kG}( \tilde{X}|_B^G, Z) \cong \Hom_{kB}(\tilde{X},Z)$ follows by Frobenius reciprocity). Now, the group $B$ acts on the $k$-vector space $\Hom_k(\tilde{X},Z)$ (given an element $b \in B$ and a $k$-vector space homomorphism $\phi \in \Hom_k(\tilde{X},Z)$, $b.\phi \in \Hom_k(\tilde{X},Z)$ is defined by $(b.\phi)(x)=b\phi(b^{-1}x)$ for any $x \in \tilde{X}$). We have an isomorphism $(\Hom_{k}(\tilde{X},Z))^B \cong \Hom_{kB}(\tilde{X},Z)$, from which it follows that $(\Hom_{k}(\tilde{X},Z))^B=0$. Since $A$ is a normal subgroup of $B$, we also have $0=(\Hom_{k}(\tilde{X},Z))^B = [ (\Hom_{k}(\tilde{X},Z))^A ]^{B/A}$. But, $B/A$ is an $r$-group and $\char(k)=r$, so this is possible only if $(\Hom_{k}(\tilde{X},Z))^A=0$ (otherwise, the $r$-group $B/A$ would have a non-zero fixed point on the $k$-vector space $(\Hom_{k}(\tilde{X},Z))^A$). Therefore, $\Hom_{kA}(\tilde{X},Z) \cong (\Hom_{k}(\tilde{X},Z))^A=0$. \\

Applying the five-term inflation-restriction exact sequence on cohomology to the $kB$-module $\tilde{X}^* \otimes_k Z$ (where $\tilde{X}^*$ is the $k$-linear dual of $\tilde{X}$), we have: \\

$0 \to H^1(B/A, (\tilde{X}^* \otimes_k Z)^A) \to H^1(B,\tilde{X}^* \otimes_k Z) \to H^1(A,\tilde{X}^* \otimes_k Z)^{B/A} \to H^2(B/A,(\tilde{X}^* \otimes_k Z)^A) \to H^2(B,\tilde{X}^* \otimes_k Z).$ \\

By assumption, $r \centernot\mid |A|$; therefore, $kA$ is semisimple by Maschke's Theorem, which means that $H^1(A,\tilde{X}^* \otimes_k Z)^{B/A}=0$. Since $(\tilde{X}^* \otimes_k Z)^A \cong \Hom_{kA}(\tilde{X},Z)=0$, we have $H^1(B/A, (\tilde{X}^* \otimes_k Z)^A)=0$. We conclude that $H^1(B,\tilde{X}^* \otimes_k Z)=0$ by exactness of the sequence above. But, $H^1(B,\tilde{X}^* \otimes_k Z) \cong \Ext^1_{kB}(k,\tilde{X}^* \otimes_k Z) \cong \Ext^1_{kB}(\tilde{X},Z)$, so it follows that $\Ext^1_{kB}(\tilde{X},Z)=0$, as needed. 
\end{proof}

\begin{proposition}\label{nofactors}
Let $Y$ be an irreducible $kG$-module in the principal series ${\rm{Irr}}_k(G | (T,X))$ and let $V$ be an irreducible $kG$-module in the principal series ${\rm{Irr}}_k(G | (T,X'))$, with \\
${\rm{Irr}}_k(G | (T,X)) \neq {\rm{Irr}}_k(G | (T,X'))$. Then, $[R_T^G(X):V]=0$.
\end{proposition}

\begin{proof}
Since $V \in {\rm{Irr}}_k(G | (T,X'))$, $V$ is a composition factor of the head and socle of $R_T^G(X')$ and ${^{*}{R}}^G_T(V) \neq 0$. Now, $W$ gives a full set of $(B,B)$-double coset representatives in $G$, so the Mackey decomposition yields
$$^{*}{R}^G_T(R_T^G(X')) \cong \underset{w \in W}{\bigoplus} {R_{^{w}{T} \cap T}^T} ^{*}{R}^{^{w}{T}}_{{^{w}{T}} \cap T} ({^{w}{X'}}).$$
But, since $W \leq N_G(T)/T$, $^{w}{T}=wTw^{-1}=T$ for any $w \in W$. Therefore, the functors $R_{^{w}{T} \cap T}^T$ and $^*{R}^{^{w}{T}}_{^{w}{T} \cap T}$ are equal to the identity functor on $kT-\text{mod}$ for all $w \in W$, and it follows that 
$$^{*}{R}^G_T(R_T^G(X')) \cong \underset{w \in W}{\bigoplus} {^{w}{X'}}.$$
Since $^{*}{R}^G_T$ is exact, $^{*}{R}^G_T(V)$ is a non-zero $kT$-submodule of the completely reducible $kT$-module $^{*}{R}^G_T(R_T^G(X'))$. Thus, there must be some subset $\Omega \subseteq W$ such that
\begin{equation}\label{isomega} 
^{*}{R}^G_T(V) \cong \underset{w \in \Omega}{\bigoplus} {^{w}{X'}}.
\end{equation}
Suppose, for contradiction, that $[R_T^G(X):V] \neq 0$. By the Mackey decomposition, 
$$^{*}{R}^G_T(R_T^G(X)) \cong \underset{w \in W}{\bigoplus} {^{w}{X}}.$$ 
Thus, if $V$ is a composition factor of $R_T^G(X)$, then ${^{*}{R}}^G_T(V)$ is a non-zero submodule of $^{*}{R}^G_T(R_T^G(X))$, which means that $^{w}{X} \subseteq {^{*}{R}}^G_T(V)$ for some $w \in W$. Then, (\ref{isomega}) yields $^{w}{X} \cong {^{w'}{X'}}$ for some $w' \in \Omega$, so that $X' \cong {^{(w')^{-1}w}{X}}$. But, if $X'$ is a twist of $X$ by an element of $W$, then ${\rm{Irr}}_k(G | (T,X))={\rm{Irr}}_k(G | (T,X'))$, contradicting the assumption in the statement of the proposition. 
\end{proof}

\begin{corollary}\label{edit1}
If $Y \in {\rm{Irr}}_k(G | (T,X))$ and $V \in {\rm{Irr}}_k(G | (T,X'))$, with ${\rm{Irr}}_k(G | (T,X)) \neq {\rm{Irr}}_k(G | (T,X'))$, then, $\Ext^1_{kG}(Y,V)=0$.
\end{corollary}

\begin{proof}
As above, let $\L^0$ be a submodule of $R_T^G(X)$ with $R_T^G(X)/\L^0 \cong Y$. Then, by Theorem \ref{headcount} and Proposition \ref{nofactors}, $\dim~\Ext^1_{kG}(Y,V)=[\head(\L^0):V] \leq [\L^0:V] \leq [R_T^G(X):V]=0$. \\
\end{proof}

\section{A Bound on the Dimension of $\Ext^1_{kG}(Y,V)$ when $Y$ and $V$ belong to the same Principal Series}

Let $G$ be a finite group of Lie type defined in characteristic $p$, and let $k$ be an algebraically closed field of characteristic $r>0$, $r \neq p$. In this section, we will assume that $Y$ and $V$ belong to the same principal series ${\rm{Irr}}_k(G | (T,X))$ (where $X$ is any one-dimensional $kT$-module). Let $e$ be the $r$-rank of the maximal torus $T$ (that is, $e$ is the maximal dimension of an elementary abelian $r$-subgroup of $T$ as an $\mathbb{F}_r$-vector space). Our goal is to prove that $\dim~ \Ext^1_{kG}(Y,V) \leq [W:W(T,X)]~|W|+{\rm{min}}(\dim~Y,\dim~V)e$, where $W(T,X)$ is the inertia group of $X$ (see Definition \ref{inertia}). In Lemma \ref{countagain}, we provide an upper bound on the number of times an irreducible representation in the principal series ${\rm{Irr}}_k(G | (T,X))$ can appear as a composition factor of the induced module $R_T^G(X)$. (The proof of Lemma \ref{countagain} uses strategies of \cite[Prop. 3.1]{gj}.) \\

\begin{lemma}\label{countagain}
If $Z$ is an irreducible $kG$-module in the principal series ${\rm{Irr}}_k(G | (T,X))$, then $[R_T^G(X):Z] \leq |W|$.
\end{lemma}

\begin{proof}
By the Mackey decomposition,
$$^{*}{R}^G_T(R_T^G(X)) \cong \underset{w \in W}{\bigoplus} {R_{^{w}{T} \cap T}^T} ^{*}{R}^{^{w}{T}}_{{^{w}{T}} \cap T} ({^{w}{X}}) \cong \underset{w \in W}{\oplus} {^{w}{X}}.$$
Since ${^{w}{X}}$ is an irreducible one-dimensional $kT$-module for every $w \in W$, the $kT$-module $^{*}{R}^G_T(R_T^G(X))$ is completely reducible and $\dim~^{*}{R}^G_T(R_T^G(X))=|W|$. On the other hand, since $^{*}{R}^G_T$ is exact, $^{*}{R}^G_T(R_T^G(X))$ has the following direct sum decomposition as a $kT$-module:
$$^{*}{R}^G_T(R_T^G(X)) \cong \underset{Z' \in {\rm{Irr}}_k(G)}{\bigoplus} ^{*}{R}^G_T(Z')^{\oplus [R_T^G(X):Z']}.$$
Therefore,
\begin{equation}\label{newyear}
\dim~^{*}{R}^G_T(R_T^G(X))=\underset{Z' \in {\rm{Irr}}_k(G)}{\sum} [R_T^G(X):Z']~ \dim~^{*}{R}^G_T(Z').
\end{equation}
Since $Z \in {\rm{Irr}}_k(G | (T,X))$, we have $^{*}{R}^G_T(Z) \neq 0$. Thus, $[R_T^G(X):Z]$ appears with non-zero coefficient in (\ref{newyear}), and it follows that $[R_T^G(X):Z] \leq \dim~^{*}{R}^G_T(R_T^G(X))=|W|$. 
\end{proof}

To establish the desired bound on $\Ext^1_{kG}(Y,V)$, we will work with the unique Sylow $r$-subgroup $T_r$ of the abelian group $T$. Hence, we will have to break our proof into two cases, depending on whether or not $|B|$ is divisible by the characteristic $r$ of $k$. 

\subsection{Case I: $r \nmid |B|$}

Assume that the characteristic $r$ of $k$ does not divide $|B|$. \\

\begin{theorem}
If $Y$ and $V$ are irreducible $kG$-modules in the principal series ${\rm{Irr}}_k(G|(T,X))$, then, $\dim~ \Ext^1_{kG}(Y,V) \leq |W|$.
\end{theorem}

\begin{proof}
Since $V \in {\rm{Irr}}_k(G|(T,X))$, $V$ is contained in the socle of $R_T^G(X)$. Therefore, we have a short exact sequence of $kG$-modules $$0 \to V \to R_T^G(X) \to M \to 0$$
(where $M \cong R_T^G(X)/V$). By assumption, $r \centernot\mid |B|$, which means that $kB$ is semisimple and $\tilde{X}$ is an injective $kB$-module. Since induction from $B$ to $G$ is exact, $R_T^G(X)=\tilde{X}|_B^G$ is an injective $kG$-module and it follows that $\Ext^1_{kG}(Y,R_T^G(X))=0$. So, the short exact sequence above induces the exact sequence
$$0 \to \Hom_{kG}(Y,V) \to \Hom_{kG}(Y,R_T^G(X)) \to \Hom_{kG}(Y,M) \to \Ext^1_{kG}(Y,V) \to 0.$$
Therefore, $\dim~\Ext^1_{kG}(Y,V) \leq \dim~ \Hom_{kG}(Y,M) =[\soc(M):Y] \leq [M:Y] \leq [R_T^G(X):Y] \leq |W|$ (where the last inequality follows by Lemma \ref{countagain}). 
\end{proof}

\subsection{Case II: $r \mid |B|$} 

Assume that $|B|$ is divisible by the characteristic $r$ of $k$. Let $T_r$ be the (unique) normal Sylow $r$-subgroup of the abelian group $T$. Since $r \mid |B|=|U||T|$ and $U$ is a $p$-group, $r \mid |T|$, so $T_r$ is a non-trivial $r$-subgroup of $T$. We will reduce the problem of bounding $\Ext^1_{kG}(Y,V)$ to a problem of bounding a cohomology group of $T_r$. To achieve this reduction, we will work with the permutation module $k|_{T_r}^T$. Since $T_r$ is a normal subgroup of $T$, $T/T_r$ is a group and $k|_{T_r}^T \cong k[T/T_r]$. The group algebra $k[T/T_r]$ is semisimple by Maschke's Theorem, which means that $k|_{T_r}^T$ is a completely reducible $kT$-module. \\

Let $Z$ be any irreducible (and, neccessarily one-dimensional) $kT$-module. Since the only irreducible module for an $r$-group in characteristic $r$ is the trivial module $k$, we must have $Z \downarrow_{T_r}=k$. So, by Frobenius reciprocity, $\Hom_{kT}(k|_{T_r}^T,Z) \cong \Hom_{kT_r}(k,Z \downarrow_{T_r})=\Hom_{kT_r}(k,k) \cong k$, which means that $[k|_{T_r}^T:Z]=1$. Thus, the completely reducible $kT$-module $k|_{T_r}^T$ contains every irreducible $kT$-module as a direct summand exactly once. \\

\begin{lemma}\label{rgroup}
Let $A$ be an abelian $r$-group or rank $e$. Then, $\dim~H^n(A,k)=\displaystyle{n+e-1 \choose n}$ for all $n \geq 0$.
\end{lemma}

\begin{proof}
We proceed by induction on $e$. If $e=1$, then $A=\Z/m\Z$ for some $m=r^d$ ($d \in \Z^+$), and $H^n(A,k) \cong k$ for all $n \geq 0$ since $A$ is a cyclic $r$-group. In particular, $\dim~H^n(A,k) = 1 = \displaystyle{n +1-1 \choose n}$, and the statement of the lemma holds. Suppose now that $e>1$ and $\dim~H^n(A',k)=\displaystyle{n+e-2 \choose n}$ for all abelian $r$-groups $A'$ of rank $e-1$ and all $n \geq 0$. Let $m_1, \ldots, m_e$ be positive integers such that $m_i=r^{d_i}$ ($d_i \in \Z^+$) for $1 \leq i \leq e$ and $A = \Z/m_1\Z \times \Z/m_2\Z \times \cdots \times \Z/m_e\Z$. By the K{\"u}nneth formula,

\begin{align*}
H^n(A,k) &=H^n(\Z/m_1\Z \times \Z/m_2\Z \times \cdots \times \Z/m_e\Z,k) \\
&\cong \overset{n}{\underset{i=0}{\bigoplus}} H^i(\Z/m_1\Z \times \cdots \times \Z/m_{e-1}\Z,k) \otimes H^{n-i}(\Z/m_e \Z,k) \\
&\cong  \overset{n}{\underset{i=0}{\bigoplus}} H^i(\Z/m_1\Z \times \cdots \times \Z/m_{e-1}\Z,k) \otimes k
\end{align*}
By the inductive hypothesis, $\dim~H^i(\Z/m_1\Z \times \cdots \times \Z/m_{e-1}\Z,k) = \displaystyle{i+e-2 \choose i}$ for $0 \leq i \leq n$. Therefore, $\dim~H^n(A,k) = \overset{n}{\underset{i=0}{\sum}} \displaystyle{i+e-2 \choose i}=\displaystyle{n+e-1 \choose n}$. \\
(Note that the last step in the computation above is a consequence of the following sum formula for binomial coefficients: $\overset{n}{\underset{i=0}{\sum}} \displaystyle{r+i \choose i} = \displaystyle{r+n +1 \choose n}$.) 
\end{proof}

\begin{lemma}\label{rsylow}
Let $V$ be a $kT_r$-module and assume that $\dim~V=n$ as a $k$-vector space. Then, $\dim~\Ext^1_{kT_r}(k,V) = ne$ (where $e$ is the $r$-rank of $T$).
\end{lemma}

\begin{proof} We proceed by induction on the dimension $n$ of $V$. Throughout this proof, we will use the fact that the only irreducible module for an $r$-group in characteristic $r$ is the trivial module $k$. When $n=1$, we have $V=k$, so we must show that $\dim~\Ext^1_{kT_r}(k,k) \leq e$. But, $T_r$ is an abelian $r$-group of rank $e$. So, by Lemma \ref{rgroup}, $\dim~\Ext^1_{kT_r}(k,k)=\dim~H^1(T_r,k) = \displaystyle{1+e-1 \choose 1} = e$. \\

Suppose now that $\dim~V=n > 1$ and that the statement of the lemma holds for all $kT_r$-modules $V'$ of dimension $n-1$. Let $0=V_0 \subset V_1 \subset V_2 \subset \cdots \subset V_n=V$ be a composition series for $V$ as a $kT_r$-module. Note that every composition factor $V_i/V_{i-1}$ ($1 \leq i \leq n$) is isomorphic to $k$. Therefore, we have a short exact sequence $0 \to V_{n-1} \to V \to k \to 0$, which induces the following long exact $\Ext$ sequence:
$$\cdots \to \Ext^1_{kT_r}(k,k) \to \Ext^1_{kT_r}(V,k) \to \Ext^1_{kT_r}(V_{n-1},k) \to \cdots.$$
Thus, $\dim~\Ext^1_{kT_r}(V,k) \leq \dim~\Ext^1_{kT_r}(k,k)+\dim~\Ext^1_{kT_r}(V_{n-1},k)$. Now, $ \dim~\Ext^1_{kT_r}(k,k) = e$ and $\dim~\Ext^1_{kT_r}(V_{n-1},k) \leq (n-1)e$ by the base case and the inductive hypothesis, respectively. Therefore, $\dim~\Ext^1_{kT_r}(V,k) \leq e+(n-1)e=ne$, as needed. 
\end{proof}

\begin{theorem}\label{bigthm}
Let $Y$ and $V$ be irreducible $kG$-modules in the principal series ${\rm{Irr}}_k(G|(T,X))$, let $W(T,X)$ be the inertia group of $X$, and let $e$ be the rank of a Sylow $r$-subgroup of $T$. In this case, 
$$\dim~ \Ext^1_{kG}(Y,V) \leq [W:W(T,X)]~|W|+(\dim~V)e.$$
\end{theorem}

\begin{proof}
Let $\widetilde{\mathcal{I}}=\{ {^{w}{X}}~ |~ w \in W\}$. For any $w \in W$, $^{w}{T}=T$, which means that the $k({^{w}T})$-module ${^{w}{X}}$ has the structure of a $kT$-module. Therefore, $\widetilde{\mathcal{I}}$ is a subset of ${\rm{Irr}}_k(T)$. Let $\mathcal{I}$ be a subset of $\widetilde{\mathcal{I}}$ consisting of one representative of each isomorphism class of irreducible $kT$-modules appearing in $\widetilde{\mathcal{I}}$. Since $W(T,X)=\{ w \in W ~|~{^{w}{X}} \cong X\}$, $|\mathcal{I}|=[W:W(T,X)]$. \\

Since every irreducible $kT$-module occurs exactly once as a composition factor of the completely reducible $kT$-module $k|_{T_r}^T$, we have a direct sum decomposition 
$$k|_{T_r}^T \cong M \oplus \Big(\underset{X' \in \mathcal{I}}{\oplus} X'\Big),$$ 
where $M$ is a completely reducible $kT$-module which does not contain $X'$ as a composition factor for any $X' \in \mathcal{I}$. Applying the (exact) Harish-Chandra induction functor $R_T^G$, we can write
\begin{equation}\label{newyear2} 
R_T^G(k|_{T_r}^T) \cong  R_T^G(M) \oplus \Big(\underset{X' \in \mathcal{I}}{\oplus} R_T^G(X') \Big).
\end{equation}
Now, since $Y \in {\rm{Irr}}_k(G|(T,X))$, $Y$ is in the head of $R_T^G(X)$; by (\ref{newyear2}), $Y$ is also in the head of $R_T^G(k|_{T_r}^T)$. Therefore, we have a short exact sequence of $kG$-modules $$0 \to M' \to R_T^G(k|_{T_r}^T) \to Y \to 0$$
(where $M'$ is a $kG$-submodule of $R_T^G(k|_{T_r}^T)$ with $R_T^G(k|_{T_r}^T)/M' \cong Y$), which gives rise to the long exact sequence
\begin{align*}
&0 \to \Hom_{kG}(Y,V) \to \Hom_{kG}(R_T^G(k|_{T_r}^T),V) \to \Hom_{kG}(M',V) \to \Ext^1_{kG}(Y,V) \to \\
&\Ext^1_{kG}(R_T^G(k|_{T_r}^T),V)\to\cdots.
\end{align*}
By exactness, $\dim~\Ext^1_{kG}(Y,V) \leq \dim~ \Hom_{kG}(M',V)+\dim~\Ext^1_{kG}(R_T^G(k|_{T_r}^T),V)$. So, to prove the theorem, it is enough to show that $\dim~ \Hom_{kG}(M',V) \leq [W:W(T,X)]~|W|$ and that $\dim~\Ext^1_{kG}(R_T^G(k|_{T_r}^T),V) \leq (\dim~V)e$. \\

First, we show that $\dim~ \Hom_{kG}(M',V) \leq [W:W(T,X)]~|W|$. We have 
$$\dim~ \Hom_{kG}(M',V)=[\head(M'):V] \leq [R_T^G(k|_{T_r}^T):V]=[R_T^G(M):V]+\underset{X' \in \mathcal{I}}{\sum} [R_T^G(X'):V].$$ 
By the independence property of Harish-Chandra induction, $R_T^G(X) \cong R_T^G(X')$ for all $X' \in \mathcal{I}$. So, using Lemma \ref{countagain}, we have that
$$\underset{X' \in \mathcal{I}}{\sum} [R_T^G(X'):V]=\underset{X' \in \mathcal{I}}{\sum} [R_T^G(X):V] \leq \underset{X' \in \mathcal{I}}{\sum} |W| \leq |\mathcal{I}|~ |W| = [W:W(T,X)]~|W|.$$ 
To show that $\dim~ \Hom_{kG}(M',V) \leq [W:W(T,X)]~|W|$, it remains to check that $[R_T^G(M):V]=0$. By assumption, $M$ has a direct decomposition $M \cong \overset{m}{\underset{i=0}{\oplus}} Z_i$ such that each $Z_i$ ($1 \leq i \leq m$) is an irreducible $kT$-module with $Z_i \centernot\cong X'$ for all $X' \in \mathcal{I}$. In particular, for any $i$ ($1 \leq i \leq m$) and for all $w \in W$, $Z_i \centernot\cong {^{w}{X}}$, which means that the Harish-Chandra series ${\rm{Irr}}_k(G|(T,Z_i))$ and ${\rm{Irr}}_k(G|(T,X))$ are distinct. By Proposition \ref{nofactors}, $[R_T^G(Z_i):V]=0$ for all $i$, and it follows that $[R_T^G(M):V]=[\overset{m}{\underset{i=0}{\oplus}} ~R_T^G(Z_i):V]=\overset{m}{\underset{i=0}{\sum}}~[R_T^G(Z_i):V]=0$. \\

We now check that $\dim~\Ext^1_{kG}(R_T^G(k|_{T_r}^T),V) \leq (\dim~V)e$. Using the definition of the functor $R_T^G$ and the Eckmann-Shapiro Lemma, we have 
$$\Ext^1_{kG}(R_T^G(k|_{T_r}^T),V) =\Ext^1_{kG}((\widetilde{k|_{T_r}^T} )|_B^G,V) \cong \Ext^1_{kB}(\widetilde{k|_{T_r}^T},V).$$ Since $B= U \rtimes T$ and $r \centernot\mid |U|=[B:T]$, the restriction map $\Ext^1_{kB}(\widetilde{k|_{T_r}^T},V) \to \Ext^1_{kT}(\widetilde{k|_{T_r}^T},V)= \Ext^1_{kT}(k|_{T_r}^T,V)$ is injective. It follows that 
$$\dim~\Ext^1_{kG}(R_T^G(k|_{T_r}^T),V)=\dim~\Ext^1_{kB}(\widetilde{k|_{T_r}^T},V) \leq \dim~\Ext^1_{kT}(k|_{T_r}^T,V).$$ Applying the Eckmann-Shapiro Lemma once more, $\Ext^1_{kT}(k|_{T_r}^T,V) \cong \Ext^1_{kT_r}(k,V)$. Therefore, $\dim~\Ext^1_{kG}(R_T^G(k|_{T_r}^T),V) \leq \dim~\Ext^1_{kT_r}(k,V) \leq (\dim~V)e$ by Lemma \ref{rsylow}. 
\end{proof}

\medskip

\begin{corollary}\label{smallerbound}
If $Y$ and $V$ are irreducible $kG$-modules in the principal series ${\rm{Irr}}_k(G|(T,X))$, then $\dim~ \Ext^1_{kG}(Y,V) \leq [W:W(T,X)]~|W|+{\rm{min}}(\dim~Y,\dim~V)e$ (where $e$ is the rank of a Sylow $r$-subgroup of $T$).
\end{corollary}

\begin{proof}
By Theorem \ref{bigthm}, $\dim~ \Ext^1_{kG}(Y,V) \leq  [W:W(T,X)]~|W|+(\dim~V)e$. Since $Y$ and $V$ belong to the principal series ${\rm{Irr}}_k(G|(T,X))$, $Y$ and $V$ occur in both the head and the socle of $R_T^G(X)$. Hence, the $k$-linear duals $Y^*$ and $V^*$ occur in both the head and the socle of $(R_T^G(X))^* \cong R_T^G(X^*)$. Thus, $Y^*$ and $V^*$ both belong to the principal series ${\rm{Irr}}_k(G|(T,X^*))$, and $\dim~ \Ext^1_{kG}(Y,V)=\dim~ \Ext^1_{kG}(V^*,Y^*) \leq  [W:W(T,X)]~|W|+(\dim~Y^*)e= [W:W(T,X)]~|W|+(\dim~Y)e$ (where the inequality $\dim~ \Ext^1_{kG}(V^*,Y^*) \leq  [W:W(T,X)]~|W|+(\dim~Y^*)e$ follows by another application of Theorem \ref{bigthm}). 
\end{proof}

\medskip
\begin{corollary}\label{shortened}
Let $Y$ and $V$ be irreducible $kG$-modules in the unipotent principal series ${\rm{Irr}}_k(G|B)$. Then, $\dim~ \Ext^1_{kG}(Y,V) \leq |W|+{\rm{min}}(\dim~Y,\dim~V)e$.
\end{corollary}

\begin{proof}
Since ${\rm{Irr}}_k(G|B)={\rm{Irr}}_k(G|(T,k))$ and $W(T,k)=W$, Corollary \ref{smallerbound} yields $$\dim~ \Ext^1_{kG}(Y,V) \leq  [W:W(T,k)]~|W|+{\rm{min}}(\dim~Y,\dim~V)e=|W|+{\rm{min}}(\dim~Y,\dim~V)e.$$
\end{proof}

Taking $Y=k$ in Corollary \ref{shortened}, we recover the bound of \cite[Cor. 3.3]{gt}.

\section{Some Preliminaries from \cite[Section 4]{geck} (The Steinberg Module and Harish-Chandra Series)}

In the next section, we will assume certain additional conditions on the group $G$ in order to find a bound on $\dim~\Ext^1_{kG}(Y,V)$ when $Y \in {\rm{Irr}}_k(G|B)$ and $V \centernot\in {\rm{Irr}}_k(G|B)$. Several key ideas behind the proof of the main theorem (Theorem \ref{boundp}) can be found in \cite{geck}, so we begin with a summary of the relevant results of \cite{geck}. \\

We define an element $\mathfrak{e} \in kG$ by $\mathfrak{e}=\underset{w \in W}{\sum} (-1)^{l(w)} n_w \b$, where $\b=\underset{b \in B}{\sum} b$ and $n_w \in N_G(T)$ is a representative of the coset $w \in W =N/T \leq N_G(T)/T$. If $k$ is a field, then $\St_k=kG \mathfrak{e}$ is the Steinberg module of $G$ \cite{steinberg}. In \cite[Sec. 4]{geck}, Geck studies the relationship between $\St_k$ and Harish-Chandra series of irreducible $kG$-modules, which allows him to determine the composition length of $\St_k$ in certain cases. While the composition length of $\St_k$ is not needed to prove any of the results presented in this paper, the Harish-Chandra series which Geck constructs with the aid of $\St_k$ are crucial to the proof of Theorem \ref{boundp}.

\subsection{An $r$-modular system.}

Let $k$ be a field of characteristic $r>0$, $r \neq p$. We will work with an $r$-modular system $(\O,K,k)$, where $\O$ is a complete discrete valuation ring with residue field $k$ and field of fractions $K$ (with $\char(K)=0$). We will assume that the fields $k$ and $K$ are both large enough, meaning that they are splitting fields for $G$ and all of its subgroups. \\

An $\O G$-module $M$ will be called a lattice if $M$ is finitely generated and free over $\O$. Given a lattice $M$, we let $KM := K \otimes_\O M$ and $\overline{M} :=k \otimes_\O M$. It is a standard fact in modular representation theory that given a projective $\O G$-lattice $M$ and an $\O G$-lattice $M'$ (with $M'$ not necessarily projective), $\dim~\Hom_{KG}(KM,KM')=\dim~\Hom_{kG}(\overline{M}, \overline{M'})$. \\

As recorded in \cite{geck} (pg. 14), Harish-Chandra induction and restriction are compatible with the $r$-modular system in the following sense. Let $J \subseteq S$ and let $L_J$ denote the Levi complement of the standard parabolic subgroup $P_J$ of $G$. Then, if $X$ is an $\O L_J$-lattice, $$KR_{L_J}^G(X) \cong R_{L_J}^G(KX) \text{ and } \overline{R_{L_J}^G(X)} \cong R_{L_J}^G(\overline{X}).$$ 
If $Y$ is an $\O G$-lattice, 
$$K {^{*}{R}}^G_{L_J}(Y) \cong {^{*}{R}}^G_{L_J}(KY) \text{ and } \overline{{^{*}{R}}^G_{L_J}(Y)} \cong {^{*}{R}}^G_{L_J}(\overline{Y}).$$

\subsection{An $\O G$-lattice which yields the Steinberg module}

Let $\St_\O$=$\O G \e$ be the Steinberg module over $\O G$ (we will call $\St_\O$ the Steinberg lattice). Then, $K \St_\O \cong \St_K$ and $\overline{\St_\O}=\St_k$. Since $\char(K)=0$, the Steinberg module $\St_K$ is irreducible. In \cite[Section 4]{geck}, Geck constructs another $\O G$-lattice which yields $\St_K$ upon base change to $K$. In the remainder of this subsection, we will describe this alternate lattice. \\

Let $\sigma: U \to K^\times$ be a group homomorphism, and define $\u_\sigma := \underset{u \in U}{\sum} \sigma(u)u \in KG$. Since $U$ is a $p$-subgroup of $G$, $r \centernot\mid |U|$ and $\sigma(u) \in \O^\times$ for all $u \in U$, which means $\u_\sigma \in \O G$. We have
$$\u_\sigma^2=\underset{u,u' \in U}{\sum} \Big(\sigma(u)u\Big)\Big(\sigma(u')u'\Big)=\underset{u \in U}{\sum} \underset{u' \in U}{\sum} \sigma(uu')uu'=\underset{u \in U}{\sum} \u_\sigma =|U| \u_\sigma.$$
Since $|U|$ is a unit in $\O$, $\frac{1}{|U|} \u_\sigma$ is an idempotent in $\O G$ and the $\O G$-lattice $\G_\sigma := \O G \u_\sigma$ is projective. \\

\begin{proposition} (\cite[Proposition 4.2]{geck}) For any group homomorphism $\sigma: U \to K^\times$, there exists a unique $\O G$-sublattice  $\G_\sigma' \subseteq \G_\sigma$ such that $K (\G_\sigma/\G_\sigma') \cong \St_K$. If $\mathscr{S}_\sigma = \G_\sigma/\G_\sigma'$, the $kG$-module $D_\sigma := \overline{\mathscr{S}_\sigma}/\rad(\overline{\mathscr{S}}_\sigma)$ is irreducible.
\end{proposition}

\subsection{The Gelfand-Graev module}

Since $G$ is a finite group of Lie type defined in characteristic $p$, there exists a connected reductive algebraic group $\mathbb{G}$ over the algebraic closure $\overline{\mathbb{F}}_p$ of the finite field $\mathbb{F}_p$ and a Steinberg endomorphism $F$ of $\mathbb{G}$ such that $G=\mathbb{G}^F$. We will now assume that the center of $\mathbb{G}$ is connected. Let $\sigma: U \to K^\times$ be a fixed regular character (see \cite[(4.3)]{geck}). In this case, the projective $\O G$-module $\G_\sigma$ is called a Gelfand-Graev module for $G$ (the Gelfand-Graev module is unique up to isomorphism when the center of $\mathbb{G}$ is connected).  Since $\G_\sigma$ is a projective $\O G$-lattice, the $r$-modular reduction $\overline{\G}_\sigma$ of $\G_\sigma$ is a projective $kG$-module. \\

For any $J \subseteq S$, let $L_J$ be the corresponding Levi subgroup of $G$. By \cite[4.3]{geck}, there is a Gelfand-Graev module $\G_\sigma^J$ for $\O L^J$. Therefore, \cite[Proposition 4.2]{geck} yields an $\O L_J$-lattice $\mathscr{S}_\sigma^J=\G_\sigma^J/(\Gamma_\sigma^J)'$ such that $D_\sigma^J := \overline{\mathscr{S}}_\sigma^J/\rad(\overline{\mathscr{S}}_\sigma^J)$ is an irreducible $kL_J$-module. \\

The $\O G$-modules $\G_\sigma$ and $\mathscr{S}_\sigma$ behave particularly well with respect to Harish-Chandra restriction. For any $J \subseteq S$, we have the following isomorphisms of $\O L_J$-modules \cite[Lemma 4.4]{geck}:
$${^{*}{R}}^G_{L_J}(\G_\sigma) \cong \G_\sigma^J \text{ and } {^{*}{R}}^G_{L_J}(\mathscr{S}_\sigma) \cong \mathscr{S}_\sigma^J.$$
(The isomorphism ${^{*}{R}}^G_{L_J}(\G_\sigma) \cong \G_\sigma^J$ is due to Rodier.) \\

Since the Harish-Chandra restriction functor is compatible with the $r$-modular system $(\O, K, k)$, we also have ${^{*}{R}}^G_{L_J}(K\G_\sigma) \cong K\G_\sigma^J$, ${^{*}{R}}^G_{L_J}(K \mathscr{S}_\sigma) \cong K\mathscr{S}_\sigma^J$, ${^{*}{R}}^G_{L_J}(\overline{\G}_\sigma) \cong \overline{\G}_\sigma^J$, and ${^{*}{R}}^G_{L_J}(\overline{\mathscr{S}}_\sigma) \cong \overline{\mathscr{S}}_\sigma^J$.

\subsection{Harish-Chandra series arising from the regular character $\sigma$}

Let $\mathscr{P}_\sigma^*=\{J \subseteq S ~|~D_\sigma^J \text{ is a cuspidal } kL_J-\text{module} \}$. For every $J \in \mathscr{P}_\sigma^*$, $D_\sigma^J$ is an irreducible cuspidal $kL_J$-module, which means that there is a Harish-Chandra series of the form ${\rm{Irr}}_k(G|(L_J,D_\sigma^J))$. \\

\begin{definition}\label{propertyp}
We will say that the pair $(G,k)$ satisfies property (P) if every composition factor of $k|_B^G$ belongs to a Harish-Chandra series of the form ${\rm{Irr}}_k(G|(L_J,D_\sigma^J))$ for some $J \in \mathscr{P}_\sigma^*$. \\
\end{definition}

When the field $k$ is clear from context we will simply say that $G$ has property (P). There are many examples of finite groups of Lie type $G$ such that $(G,k)$ has property (P). If $q$ is a power of $p$, then the pair $(\GL_n(q),k)$ has property (P) \cite[Example 4.9]{geck}. If $r$ is a linear prime,\footnote{$r$ is a linear prime for $SO_n(q)$ and $Sp_n(q)$ if $q^{i-1} \centernot\equiv -1 ~\mod~ r$ for all $i \geq 1$.} then the following pairs have property (P) \cite[4.14]{geck}:

\begin{enumerate}
\item
$(SO_n(q),k)$, $n$ odd and $q$ odd, and
\item
$(Sp_n(q),k)$, $n$ even and $q$ a power of 2.
\end{enumerate}

\section{A Bound on the Dimension of $\Ext^1$ between a Unipotent Principal Series Representation and an Irreducible Outside the Unipotent Principal Series}\label{outside}

Let $G$ be a finite group of Lie type defined in characteristic $p$, and let $k$ be an algebraically closed field of characteristic $r>0$, $r \neq p$. In this section, we will find a bound on $\dim~\Ext^1_{kG}(Y,V)$ when the pair $(G,k)$ has property (P) and $Y$ and $V$ are irreducible $kG$-modules such that $Y \in {\rm{Irr}}_k(G|B)$ and $V \centernot\in {\rm{Irr}}_k(G|B)$. By Theorem \ref{headcount}, $\Ext^1_{kG}(Y,V)=0$ when $V$ is not a composition factor of $k|_B^G$. So, it suffices to find a bound on $\dim~\Ext^1_{kG}(Y,V)$ in the case that $V$ is a composition factor of $k|_B^G$. Since $(G,k)$ has property (P), we can assume that the irreducible $kG$-module $V$ (which is a composition factor of $k|_B^G$) belongs to a Harish-Chandra series of the form ${\rm{Irr}}_k(G|(L_J,D_\sigma^J))$ for some $J \in \mathscr{P}_\sigma^*$. The next theorem gives our bound on $\dim~\Ext^1_{kG}(Y,V)$; we note that some of the proof strategies of Theorem \ref{boundp} were inspired by \cite[Prop. 4.6]{geck}. \\

\begin{theorem}\label{boundp}
Suppose that the pair $(G,k)$ has property (P), and let $Y$ and $V$ be irreducible $kG$-modules such that $Y \in {\rm{Irr}}_k(G|B)$ and $V \centernot\in {\rm{Irr}}_k(G|B)$. Assume that $V$ is a composition factor of $k|_B^G$ and that $V$ belongs to the Harish-Chandra series ${\rm{Irr}}_k(G|(L_J,D_\sigma^J))$ ($J \in \mathscr{P}_\sigma^*$). Then, $\dim~\Ext^1_{kG}(Y,V) \leq [W:W_J]$, where $W_J$ is the parabolic subgroup of $W$ generated by $J$.
\end{theorem}

\begin{proof}
By Theorem \ref{headcount}, $\dim~\Ext^1_{kG}(Y,V) \leq [k|_B^G:V]$, so it suffices to prove that $[k|_B^G:V] \leq [W:W_J]$. Since $V \in {\rm{Irr}}_k(G|(L_J,D_\sigma^J))$, $V$ is in the head of the $kG$-module $R_{L_J}^G(D_\sigma^J)$. By definition, $D_\sigma^J$ is in the head of the $r$-modular reduction $\overline{\G}^J_\sigma$ of the Gelfand-Graev module $\G^J_\sigma$ for $L_J$, which means that there is a surjective $kL_J$-module homomorphism $\overline{\G}^J_\sigma \twoheadrightarrow D_\sigma^J$. Since the Harish-Chandra induction functor $R_{L_J}^G$ is exact, there is a surjective $kG$-module homomorphism $R_{L_J}^G(\overline{\G}^J_\sigma) \twoheadrightarrow R_{L_J}^G(D_\sigma^J)$, and it follows that $V \subseteq \head \Big(R_{L_J}^G(\overline{\G}^J_\sigma) \Big)$. \\ 

Let $P_V$ denote the projective indecomposable $kG$-module with $\head(P_V)=V$. Since $\overline{\G}^J_\sigma$ is a projective $kL_J$-module and $R_{L_J}^G$ is exact, $R_{L_J}^G(\overline{\G}^J_\sigma)$ is a projective $kG$-module. Thus, $P_V$ is a direct summand of $R_{L_J}^G(\overline{\G}^J_\sigma)$ and we have
$$[k|_B^G:V]=\dim~\Hom_{kG}(P_V,k|_B^G) \leq \dim~\Hom_{kG}(R_{L_J}^G(\overline{\G}^J_\sigma),k|_B^G).$$
Now, $R_{L_J}^G(\overline{\G}^J_\sigma) \cong \overline{R_{L_J}^G(\G^J_\sigma)}$ is the reduction of the $\O G$-lattice $R_{L_J}^G(\G^J_\sigma)$, and $k|_B^G \cong \overline{\O |_B^G}$ is the reduction of the $\O G$-lattice $\O |_B^G$. So, since $KR_{L_J}^G(\G^J_\sigma) \cong R_{L_J}^G(K \G^J_\sigma)$ and $K \O|_B^G \cong K|_B^G=R_T^G(K)$, we have
$$\dim~\Hom_{kG}(R_{L_J}^G(\overline{\G}^J_\sigma),k|_B^G)=\dim~\Hom_{KG}(R_{L_J}^G(K\G^J_\sigma),K|_B^G)=\dim~\Hom_{KG}(R_{L_J}^G(K\G^J_\sigma),R_T^G(K)).$$
(The first equality above holds because $R_{L_J}^G(\G^J_\sigma)$ is a projective $\O G$-lattice.) \\

We will now compute $\dim~\Hom_{KG}(R_{L_J}^G(K\G^J_\sigma),R_T^G(K))$. First, since the functor $^{*}{R}^G_{L_J}$ is right adjoint to $R_{L_J}^G$, we have $$\Hom_{KG}(R_{L_J}^G(K\G^J_\sigma),R_T^G(K)) \cong \Hom_{KL_J}(K\G^J_\sigma,{^{*}{R}}^G_{L_J}R_T^G(K)).$$
Let $P_J$ be the standard parabolic subgroup of $G$ containing $L_J$, and let $^{J}{W}$ denote the set of shortest right coset representatives of $W_J$ in $W$. Then, $^{J}{W}$ gives a full set of $(P_J,B)$-double coset representatives in $G$ and it follows by the Mackey decomposition that

\begin{align*}
\Hom_{KL_J}(K\G^J_\sigma,{^{*}{R}}^G_{L_J}R_T^G(K)) &\cong \Hom_{KL_J}(K\G^J_\sigma, \underset{w \in {^{J}{W}}}{\oplus} R_{{^{w}{T}} \cap L_J}^{L_J} {^{*}{R}}^{{^{w}{T}}}_{{^{w}{T}} \cap L_J}(^{w}{K})) \\
&\cong \underset{w \in {^{J}{W}}}{\oplus} \Hom_{KL_J}(K\G^J_\sigma, R_{{^{w}{T}} \cap L_J}^{L_J} {^{*}{R}}^{{^{w}{T}}}_{{^{w}{T}} \cap L_J}(^{w}{K})).
\end{align*}
Since $K$ is the trivial one-dimensional $KT$-module, $^{w}{K}$ is the trivial one-dimensional $K({^{w}{T}})$-module. For all $w \in W$, $^{w}{T}=T$, $ {^{*}{R}}^{{^{w}{T}}}_{{^{w}{T}} \cap L_J}={^{*}{R}}^T_{T \cap L_J}={^{*}{R}}^T_T$ is the identity functor on $KT$-mod, $^{w}{K}=K$, and $R_{{^{w}{T}} \cap L_J}^{L_J}=R_T^{L_J}$. Therefore, continuing the calculation above (and using the adjointness of $R_T^{L_J}$ and ${^{*}{R}}^{L_J}_T$), we have
\begin{align*}
\underset{w \in {^{J}{W}}}{\oplus} \Hom_{KL_J}(K\G^J_\sigma, R_{{^{w}{T}} \cap L_J}^{L_J} {^{*}{R}}^{{^{w}{T}}}_{{^{w}{T}} \cap L_J}(^{w}{K})) &\cong \underset{w \in {^{J}{W}}}{\oplus} \Hom_{KL_J}(K\G^J_\sigma, R_T^{L_J}(K)) \\
&\cong \underset{w \in {^{J}{W}}}{\oplus} \Hom_{KT} ({^{*}{R}}^{L_J}_T(K\G^J_\sigma), K).
\end{align*}

Now, since $T=L_{\emptyset}$, ${^{*}{R}}^{L_J}_T(K\G^J_\sigma) \cong K \G^{\emptyset}_\sigma$ by Rodier's result on the restriction of Gelfand-Graev modules in characteristic $0$. But, $T$ has a trivial unipotent radical, which means that the Gelfand-Graev module $K \G^{\emptyset}_\sigma$ for $T$ is equal to the group algebra $KT$. Therefore,
$$\underset{w \in {^{J}{W}}}{\oplus} \Hom_{KT} ({^{*}{R}}^{L_J}_T(K\G^J_\sigma), K) \cong \underset{w \in {^{J}{W}}}{\oplus} \Hom_{KT} (KT, K).$$
$K$ appears once as a direct summand of the completely reducible $KT$-module $KT$, so $\dim~\Hom_{KT} (KT, K)=1$ and, following our chain of calculations, we have 
$$\dim~\Hom_{KG}(R_{L_J}^G(K\G^J_\sigma),R_T^G(K))=|{^{J}{W}}|.$$ 
Thus, $[k|_B^G:V] \leq \dim~\Hom_{kG}(R_{L_J}^G(\overline{\G}^J_\sigma),k|_B^G)=\dim~\Hom_{KG}(R_{L_J}^G(K\G^J_\sigma),R_T^G(K))=|{^{J}{W}}| = [W:W_J]$, as needed. 
\end{proof}

The bound of Theorem \ref{boundp} is particularly strong in the case that $V$ is a cuspidal irreducible $kG$-module. \\

\begin{corollary}
Suppose that the pair $(G,k)$ has property (P), and let $V$ be a cuspidal irreducible $kG$-module. Then, $\dim~\Ext^1_{kG}(Y,V) \leq 1$ for any irreducible $kG$-module $Y \in {\rm{Irr}}_k(G|B)$.
\end{corollary}

\begin{proof}
If $V$ is not a composition factor of $k|_B^G$, then $\Ext^1_{kG}(Y,V)=0$ by Theorem \ref{headcount}. If $V$ is a composition factor of $k|_B^G$, the Harish-Chandra series containing $V$ is of the form $\{ V \} = \text{Irr}_k(G|(G,D_\sigma^S))$. So, by Theorem \ref{boundp}, $\dim~\Ext^1_{kG}(Y,V) \leq [W:W]=1$. 
\end{proof}

\section{Explicit Computations of Bounds on the Dimension of $\Ext^1$ Between Irreducible Modules for $\GL_n(q)$ in Cross Characteristic}

We will explicitly demonstrate the bounds on $\dim~ \Ext^1_{kG}(Y,V)$ given by Theorem \ref{boundp} in a series of examples. In these examples, we will work with the general linear group $G=\GL_n(q)$ over the finite field $\mathbb{F}_q$, where $q$ is a power of a prime $p$. Let $k$ be an algebraically closed field of characteristic $r>0$, $r \neq p$. By \cite[4.9]{geck}, the pair $(\GL_n(q),k)$ satisfies property (P), and consequently the bounds of Theorem \ref{boundp} apply. \\

We will use the parameterization of $k\GL_n(q)$-modules given by Dipper and James in \cite[(3.1)]{dj}. In this labeling, the irreducible constituents of $k|_B^G$ are given by $D(1,\lambda)$, where $\lambda$ ranges over the partitions of $n$. (The trivial irreducible $kG$-module is parameterized as $k=D(1,(n))$.) We will present an algorithm of Dipper and Du \cite[Sec.~4.3]{dd} which determines the Harish-Chandra vertex of any irreducible $D(1,\lambda)$, $\lambda \vdash n$. Then, we will apply Dipper and Du's algorithm to certain irreducible $k\GL_n(q)$-modules in order to obtain explicit bounds on the dimension of $\Ext^1$ between these irreducibles. \\

\subsection{The Harish-Chandra Vertex of the Module $D(1,\l)$ \cite{dd}}

Let $|q \pmod{r}|$ denote the multiplicative order of $q$ modulo $r$, and define an integer $l \in \Z^+$ by
$$l=
\begin{cases}
r &\text{ if } |q \pmod{r}|=1 \\
|q \pmod{r}| &\text{ if } |q \pmod{r}|>1.
\end{cases}
$$
By \cite[Thm. 7.6]{djgeneral}, ${\rm{Irr}}_k(G|B)=\{ D(1,\l) | \l \text{ is } l\text{-regular} \}$. Thus, when $\l$ is $l$-regular, the Harish Chandra vertex of $\l$ is the maximal torus $T$. In general, the Harish-Chandra vertex of an irreducible $kG$-module $D(1,\l)$ (where $\l$ is any partition of $n$) may be determined as follows. \\

Let $\l'$ denote the dual partition of $\l$. Suppose that $\l'=\l'_{-1} + l \l'_0+lr \l'_1 + lr^2 \l'_2+\cdots$ is the $l-r$-adic decomposition of $\l'$, meaning that $\l'_{-1} \vdash n_{-1}$ is $l$-restricted and $\l'_a \vdash n_a$ is $r$-restricted for $a \geq 0$. \footnote{A partition $\mu \vdash n$ is called $l$-restricted if its dual $\mu'$ is $l$-regular, meaning that every part of $\mu'$ occurs less than $l$ times.} Since $\l' \vdash n$, the integers $n_a \in \Z_{\geq 0}$ must satisfy the relation
$$n=n_{-1}+\underset{a \geq 0}{\sum} lr^a n_a.$$
If $\l'$ has the $l-r$-adic decomposition given above, a Harish-Chandra vertex of the irreducible $k\GL_n(q)$-module $D(1,\l)$ is the Levi subgroup 
$$L=\GL_1(q)^{\times n_{-1}} \times \GL_l(q)^{\times n_0} \times  \GL_{lr}(q)^{\times n_1} \times \GL_{lr^2}(q)^{\times n_2} \times \cdots$$
of $\GL_n(q)$. \\

\begin{example}
Let $\char(k)=r=2$, let $G=\GL_2(q)$, and assume that $2 \centernot\mid q$. The only partitions of 2 are $(2)$ and $(1,1)=(1^2)$, which means that the irreducible $kG$-modules which appear as composition factors of $k|_B^G$ are $D(1,(2))$ and $D(1,(1^2))$. Here, $D(1,(2))=k \in {\rm{Irr}}_k(G|B)={\rm{Irr}}_k(G|(T,k))$. Therefore, the Harish-Chandra vertex of $D(1,(2))$ is the maximal torus $T$. \\

We will apply Dipper and Du's algorithm to find a Harish-Chandra vertex of $D(1,(1^2))$. Since $r=2$ and $2 \centernot\mid q$, $|q \pmod{r}|=|q \pmod{2}|=1$. Therefore, we set $l=r=2$. The $2-2$-adic decomposition of $(1^2)'=(2)$ is $(2)=(0)+2(1)=2(1)$. Since $\l_0=(1) \vdash 1$, $n_0=1$; since $\l_a = (0)$ for all $a \neq 0$, $n_a=0$ for $a \neq 0$. Thus, the Harish-Chandra vertex of $D(1,(1^2))$ is $\GL_2(q)=G$, which means that $D(1,(1^2))$ is cuspidal. \\
\end{example}

\subsection{Some Examples of Bounds on the Dimension of $\Ext^1$}

As above, let $G=\GL_n(q)$ and let $k$ be an algebraically closed field of characteristic $r>0$, $r \centernot\mid q$. In this context, Theorem \ref{boundp} may be restated as follows. \\

\begin{theorem}\label{bound} Let $\l \vdash n$ and assume that the irreducible $kG$-module $D(1,\l)$ belongs to the unipotent principal series ${\rm{Irr}}_k(G|B)$. Suppose that $\mu \vdash n$ is such that $D(1,\mu)$ is not a composition factor of $k|_B^G$. If  $D(1,\mu)$ has Harish-Chandra vertex $L_J \neq T$ ($\emptyset \neq J \subseteq S$), then $\dim~\Ext^1_{kG}(D(1,\l),D(1,\mu)) \leq \dfrac{|W|}{|W_J|}$. \footnote{In the case that $\l=(n)$ and $D(1,\l)=k$, this bound can be improved. As was shown by Guralnick and Tiep, $\dim~\Ext^1_{kG}(k,D(1,\mu))=\dim~H^1(G,D(1,\mu)) \leq 1$ when $D(1,\mu) \notin {\rm{Irr}}_k(G|B)$.}

\end{theorem}

In the remainder of this section, we will provide several explicit examples of the bound given in Theorem \ref{bound}. \\

\begin{example}\label{gl3} $G=\GL_3(q)$, $\char(k) = r >0$, $r \centernot\mid q$ \\

In this case, $W=\mathfrak{S}_3$, the symmetric group on the set $\{1,2,3\}$, and $S$ is the set $\{(1,2), (2,3) \}$ of fundamental reflections in $W$. There are three partitions of 3: $(3)$, $(2,1)$, and $(1^3)$. So, the irreducible $kG$-modules which occur as composition factors of $k|_B^G$ are $D(1,(3))$, $D(1,(2,1))$, and $D(1,(1^3))$. Since $l >1$ by definition, the partitions $(3)$ and $(2,1)$ are $l$-regular for any $l$. Therefore, $D(1,(3))$ and $D(1,(2,1))$ belong to ${\rm{Irr}}_k(G|B)$ for any $r$ with $r \centernot\mid q$.

Thus, $D(1,(1^3))$ is the only composition factor of $k|_B^G$ whose Harish-Chandra vertex varies with $r$. We will compute the Harish-Chandra vertex of $D(1,(1^3))$ and find the corresponding bound on $\Ext^1$ in the cases of $r=2$ and $r=3$. \\

\begin{enumerate}
\item
$\GL_3(q)$, $r=2 \centernot\mid q$ \\

Since $r=2 \centernot\mid q$, $|q \pmod{2}|=1$. Thus, we set $l=r=2$. The $2-2$-adic decomposition of $(3)$ is $(3)=(1)+2(1)$. So, a Harish-Chandra vertex of $D(1,(1^3))$ is $L=\GL_1(q) \times \GL_2(q)$, which has a Weyl group of order 2. So, Theorem \ref{bound} yields the following bounds:

$\dim~\Ext^1_{kG}\Big(k,D(1,(1^3))\Big) \leq \dfrac{|W|}{2}=\dfrac{3!}{2}=3, \text{ and}$ \\
$\dim~\Ext^1_{kG}\Big(D(1,(2,1)),D(1,(1^3))\Big) \leq 3.$ \\

\item
$\GL_3(q),$ $r=3 \centernot\mid q$ \\

In this case $|q \pmod 3|$ is equal to 1 or 2. If $|q \pmod 3|=1$, then $l=3$ and the approach used above shows that $D(1, (1^3))$ is cuspidal. Thus, Theorem \ref{bound} yields the following bounds: 

$\dim~\Ext^1_{kG} \Big(k,D(1,(1^3))\Big) \leq \dfrac{|W|}{|W|}=1, \text{ and}$ \\
$\dim~\Ext^1_{kG}\Big(D(1,(2,1)),D(1,(1^3))\Big) \leq 1$. \\

If $|q \pmod{3}|=2$, then $l=2$ and our algorithm shows that a Harish-Chandra vertex of $D(1,(1^3))$ is $\GL_2(q) \times \GL_1(q)$. In this case, Theorem \ref{bound} yields the bounds:

$\dim~\Ext^1_{kG} \Big(k,D(1,(1^3))\Big) \leq 3, \text{ and}$ \\
$\dim~\Ext^1_{kG}\Big(D(1,(2,1)),D(1,(1^3))\Big) \leq 3$. \\
\end{enumerate}

{\bf{Remark.}} Guralnick and Tiep's results yield sharper bounds on $\dim~\Ext^1_{kG} \Big(k,D(1,(1^3))\Big)$ in the cases of $G=\GL_3(q)$, $r=2$ and $G=\GL_3(q)$, $r=3$, $|q \pmod{3}|=2$; by \cite[Cor. 6.5]{gt}, $\dim~ \Ext^1_{kG} \Big(k,D(1,(1^3))\Big) \leq 1$ whenever $0 \neq r \neq p$. \\

\end{example}

\begin{example}\label{gl4}
$G=\GL_4(q)$, $r=2 \centernot\mid q$ \\

Here, $W=\mathfrak{S}_4$, and $l=2$. The partitions of $(4)$ are: $(4)$, $(3,1)$, $(2^2)$, $(2,1^2)$, and $(1^4)$. Thus, the composition factors of $k|_B^G$ are $D(1,(4))$, $D(1,(3,1))$, $D(1,(2^2))$, $D(1,(2,1^2))$, and $D(1,(1^4))$. The approach of the previous example yields the following results. \\

${\rm{Irr}}_k(G|B)=\{ D(1,(4))=k, D(1,(3,1)) \}$

\begin{center}
\begin{tabular}{c|c|c|c|c|c}
$\l \vdash n$ & $2-2$-adic decomposition of $\l$ & Harish-Chandra vertex $L_J$ of $D(1,\l')$ & $|W_J|$ & $\dfrac{|W|}{|W_J|}$ \\
\hline
$(4)$ & $(4)=4(1)$ & $G$ & $4!$ & $4!/4!=1$ \\
$(3,1)$ & $(3,1)=(1,1)+2(1)$ & $\GL_2(q) \times \GL_1(q)^{\times 2}$ & $2$ & $4!/2=12$ \\
$(2^2)$ & $(2^2)=2(1^2)$ & $\GL_2(q)^{\times 2}$ & $4$ & $4!/4=6$ \\
\end{tabular}
\end{center}
\bigskip
Thus, Theorem \ref{bound} (2) yields the bounds: \\
$\dim~\Ext^1_{kG} \Big(k,D(1,(1^4))\Big) \leq 1,$ \\
$\dim~\Ext^1_{kG} \Big(k,D(1,(2,1^2))\Big) \leq 12,$ \\
$\dim~\Ext^1_{kG} \Big(k,D(1,(2^2))\Big) \leq 6,$ \\
$\dim~\Ext^1_{kG}\Big(D(1,(3,1)),D(1,(1^4))\Big) \leq 1$, \\
$\dim~\Ext^1_{kG}\Big(D(1,(3,1)),D(1,(2,1^2))\Big) \leq 12,$ and \\
$\dim~\Ext^1_{kG}\Big(D(1,(3,1)),D(1,(2^2))\Big) \leq 6.$
\end{example}

{\bf{Remark.}} By \cite[Cor. 6.5]{gt}, $\dim~\Ext^1_{kG} \Big(k,D(1,(2,1^2))\Big) \leq 1$ and $\dim~\Ext^1_{kG} \Big(k,D(1,(2^2))\Big) \leq 1$. However, the other four bounds cannot be improved using Guralnick and Tiep's results.

\subsection{Improved bounds for $\GL_n(q)$, $n \leq 10$}

In the case that $G=\GL_n(q)$ (with $n \leq 10$), we can use James's decomposition matrices \cite[Appendix 1]{james} to obtain sharper bounds on $\dim~\Ext^1_{kG}(Y,V)$, where $Y$ and $V$ are irreducible $kG$-modules with $Y \in {\rm{Irr}}_k(G|B)$ and $V \not\in {\rm{Irr}}_k(G|B)$. \\

Dipper and James (\cite{dj}, etc.) associate two indecomposable $kG$-modules to every partition of $n$. Given $\l \vdash n$, there is an indecomposable $kG$-module $S(1,\l)$ with $\head(S(1,\l))=D(1,\l)$, which maps to a Specht module for a Hecke algebra under an appropriate Hecke functor \cite[(3.1)]{dj}. There is also an indecomposable $kG$-module $X(1,\l)$ (called a Young module) which contains $S(1,\l)$ \cite[pg.~43]{dj}. Each indecomposable direct summand of the permutation module $k|_B^G$ is isomorphic to a Young module $X(1,\l)$ for some $\l \vdash n$. By \cite[Thm.~1]{sawada}, each Young module $X(1,\l)$ ($\l \vdash n$) has an irreducible head and socle. \\

As above, let $$l=
\begin{cases}
r &\text{ if } |q \pmod{r}|=1 \\
|q \pmod{r}| &\text{ if } |q \pmod{r}|>1.
\end{cases}
$$

\begin{proposition}\label{matrixbound}
If $\sigma \vdash n$ is $l$-regular and $\mu \vdash n$ is $l$-restricted, then $$\dim~\Ext^1_{kG}(D(1,\sigma),D(1,\mu)) \leq {\rm{max}}  (\{ [X(1,\l):D(1,\mu)] ~|~ \l \vdash n \}).$$
\end{proposition}

\begin{proof}
Since $\sigma$ is $l$-regular, $D(1,\sigma) \in {\rm{Irr}}_k(G|B)$, which means that $D(1,\sigma)$ is in the head of $k|_B^G$. Thus, there is a partition $\tilde{\sigma}$ of $n$ such that the Young module $X(1,\tilde{\sigma})$ is a direct summand of $k|_B^G$ and $\head(X(1,\tilde{\sigma}))=D(1,\sigma)$. Let $M$ be a maximal submodule of $X(1,\tilde{\sigma})$ with $X(1,\tilde{\sigma}) / M \cong D(1,\sigma)$ and write $k|_B^G = X(1,\tilde{\sigma}) \oplus X$, where the $kG$-module $X$ is a direct sum of various Young modules. Then, $\L^0 := M \oplus X$ is a maximal submodule of $k|_B^G$ with $k|_B^G/ \L^0 \cong D(1,\sigma)$. By Theorem \ref{headcount}, $\dim ~\Ext^1_{kG}(D(1,\sigma), D(1,\mu))= [\head(\L^0):D(1,\mu)]$. But, since $D(1,\mu) \not\in {\rm{Irr}}_k(G|B)$, $D(1,\mu)$ is not in the head of $k|_B^G$, and it follows that $[\head(\L^0):D(1,\mu)] =[\head(M):D(1,\mu)] \leq [X(1,\tilde{\sigma}):D(1,\mu)] \leq {\rm{max}}  (\{ [X(1,\l):D(1,\mu)] ~|~ \l \vdash n \})$. 
\end{proof}

\bigskip

Following \cite{james}, let $d_{\l \mu}= [S(1,\l):D(1,\mu)]$ for any partitions $\l$, $\mu$ of $n$. James \cite{james} shows how to find $d_{\l \mu}$ for $\GL_n(q)$ when $n \leq 10$ and records the integers $d_{\l \mu}$ in the matrices $\Delta_n$ in \cite[App.~1]{james}. Since the composition  factors of $X(1,\l)$ are the same as the composition factors of $\underset{\mu \vdash n}{\bigoplus} S(1,\mu)^{\oplus d_{\mu' \l'}}$, with multiplicity (see \cite[(3.4)]{james}), we can use James's matrices to compute $[X(1,\l):D(1,\mu)]$ for any $\l, \mu \vdash n$ when $n \leq 10$. As illustrated in the next example, such computations (combined with the result of Proposition \ref{matrixbound}) can yield bounds on the dimensions of $\Ext^1$ groups between irreducible $kG$-modules. \\

\begin{example}\label{matrices}
As in the first part of Example \ref{gl3}, we consider $G=\GL_3(q)$ with $r=2 \centernot\mid q$. In this case, $l=2$ ($l$ is denoted by $e$ in \cite{james}). The partitions of $3$ are $(3)$, $(2,1)$, and $1^3$; in this case, ${\rm{Irr}}_k(G|B)=\{D(1,(3)), D(1,(2,1)) \}$ and $D(1,(1^3))$ is the only irreducible $kG$-module outside of ${\rm{Irr}}_k(G|B)$. \\

We will use the information contained in the matrix $\Delta_3$ on pg.~253 of \cite{james} to find $[X(1,(1^3)):D(1,(1^3))]$, $[X(1,(2,1)):D(1,(1^3))]$, and $[X(1,(3)):D(1,(1^3))]$. We start with $[X(1,(1^3)):D(1,(1^3))]$. Since $d_{(3)(3)}=1$, $d_{(2,1)(3)}=0$ and $d_{(1^3)(3)}=1$, $X(1,(1^3))$ has the same composition factors as $S(1,(1^3)) \oplus S(1,(3))$. Therefore, $[X(1,(1^3)):D(1,(1^3))]=[S(1,(1^3)):D(1,(1^3))]+[S(1,(3)):D(1,(1^3))]=d_{(1^3)(1^3)}+d_{(3)(1^3)}=1+0=1$. The same approach yields $[X(1,(2,1)):D(1,(1^3))]=0$ and $[X(1,(3)):D(1,(1^3))]=0$ (the fact that $[X(1,(3)):D(1,(1^3))]=0$ is obvious since $X(1,(3))=S(1,(3))=D(1,(3))=k$). Thus, ${\rm{max}}  (\{ [X(1,\l):D(1,(1^3))] | \l \vdash 3 \})=1$ and Proposition \ref{matrixbound} yields the improved bounds $\dim~ \Ext^1_{kG}(k,D(1,(1^3))) \leq 1$ and $\dim ~\Ext^1_{kG}(D(1,(2,1)),D(1,(1^3))) \leq 1$.
\end{example}

\begin{example}\label{matrixex}
We use the approach of Example \ref{matrices} to obtain new bounds in the other cases considered in Examples \ref{gl3} and \ref{gl4}. \\

$\GL_3(q)$, $r=3 \centernot\mid q$, $|q \pmod{3}|=1$ \\

In this case, $l=3$, so ${\rm{Irr}}_k(G|B)=\{D(1,(3)), D(1,(2,1)) \}$ and $D(1,(1^3))$ is the only irreducible $kG$-module outside of ${\rm{Irr}}_k(G|B)$. Using the matrix $\Delta_3$ \cite[pg.~258]{james}, we find that $[X(1,(1^3)):D(1,(1^3))]=1$, $[X(1,(2,1)):D(1,(1^3))]=0$, and $[X(1,(3)):D(1,(1^3))]=0$. Therefore, Proposition \ref{matrixbound} yields the bounds: \\ 
$\dim~ \Ext^1_{kG}(D(1,(3)),D(1,(1^3))) \leq 1$ and \\
$\dim~ \Ext^1_{kG}(D(1,(2,1)),D(1,(1^3))) \leq 1$. \\

$\GL_3(q)$, $r=3 \centernot\mid q$, $|q \pmod{3}|=2$ \\

In this case, $l=2$, so we use the matrix $\Delta_3$ \cite[pg.~253]{james} to obtain the bounds: \\
$\dim~ \Ext^1_{kG}(D(1,(3)),D(1,(1^3))) \leq 1$ and \\
$\dim~ \Ext^1_{kG}(D(1,(2,1)),D(1,(1^3))) \leq 1$. \\

$\GL_4(q)$, $r=2 \centernot\mid q$ \\

The partitions of $(4)$ are $(4)$, $(3,1)$, $(2^2)$, $(2,1^2)$, and $(1^4)$. Since $l=2$, $(4)$ and $(3,1)$ are the only $l$-regular partitions of $4$. We have ${\rm{Irr}}_k(G|B)=\{ D(1,(4))=k, D(1,(3,1)) \}$; the composition factors of $k|_B^G$ outside the unipotent principal series are $D(1,(2^2))$, $D(1,(2,1^2))$, and $D(1,(1^4))$. The matrix $\Delta_4$ \cite[pg.~253]{james} (along with the appropriate adjustments for $r=2$ described at the bottom of pg. 253) yields the following information about composition multiplicities of $D(1,(2^2))$, $D(1,(2,1^2))$, and $D(1,(1^4))$ in the Young modules $X(1,\l)$, $\l \vdash 4$. \\

\begin{center}
\begin{tabular}{c|c|c|c}
$X(1,\l)$ & $[X(1,\l):D(1,(2^2))]$ & $[X(1,\l):D(1,(2,1^2))]$ & $[X(1,\l):D(1,(1^4))]$ \\
\hline
$X(1,(1^4))$ & 0& 0 & 1  \\
$X(1,(2,1^2))$ & 1 & 1 &0 \\
$X(1,(2^2))$ & 1 & 0 &0 \\
$X(1,(3,1))$ & 0 & 0&0 \\
$X(1,(4))$ & 0 &0 &0 \\
\end{tabular}
\end{center}

Thus, Proposition \ref{matrixbound} gives the following improved bounds:  \\
$\dim~\Ext^1_{kG} \Big(k,D(1,(1^4))\Big) \leq 1,$ \\
$\dim~\Ext^1_{kG} \Big(k,D(1,(2,1^2))\Big) \leq 1,$ \\
$\dim~\Ext^1_{kG} \Big(k,D(1,(2^2))\Big) \leq 1,$ \\
$\dim~\Ext^1_{kG}\Big(D(1,(3,1)),D(1,(1^4))\Big) \leq 1$, \\
$\dim~\Ext^1_{kG}\Big(D(1,(3,1)),D(1,(2,1^2))\Big) \leq 1,$ and \\
$\dim~\Ext^1_{kG}\Big(D(1,(3,1)),D(1,(2^2))\Big) \leq 1.$ \\
\end{example}

{\bf{Remark.}} Though the results of Example \ref{matrixex} suggest that the dimensions of $\Ext^1$ groups are bounded by $1$ when $G=\GL_n(q)$ and $\char(k)=r \centernot\mid q$, this is not necessarily always the case. For instance, when $G=\GL_4(q)$, $r>2$, and $e=3$, Proposition \ref{matrixbound} yields the bound $\dim~\Ext^1_{kG}\Big(D(1,(3,1)),D(1,(2^2))\Big) \leq 2$. We will explore whether the bounds obtained via Proposition \ref{matrixbound} are optimal in future work. However, even if the bounds of Proposition \ref{matrixbound} are not optimal, they do support the widely held belief that $\Ext$ groups are ``small' in non-defining characteristic.

\section{Conclusion and Outlook}
We have made some progress toward a more complete understanding of $\Ext$ groups between irreducible $kG$-modules in non-defining characteristic. But, perhaps even more importantly, we have demonstrated that modular Harish-Chandra theory is a useful tool in the study of $\Ext$ groups between irreducible $kG$-modules. There are many open problems which may be viewed through the lens of modular Harish-Chandra theory, several of which are briefly outlined below. \\

First, there are more cases in which to compute bounds on $\dim~\Ext^1_{kG}(Y,V)$ (where $Y$, $V$ are irreducible $kG$-modules.) In this paper, we have assumed that $Y \in {\rm{Irr}}_k(G|(T,X))$ is a principal series representation. But, what if $Y$ belongs to a Harish-Chandra series of the form ${\rm{Irr}}_k(G|(L,X))$ where $T \subsetneq L$? A natural question to ask is whether we can find bounds on $\dim~\Ext^1_{kG}(Y,V)$ analogous to those of Sections 4, 5 and 7 when $Y$ is not a principal series representation. Another question to consider is whether we can drop some of the additional assumptions on $G$ in Section 7. Specifically, is there a bound analogous to that found in Section 7 when the pair $(G,k)$ does not satisfy property (P)? Additionally, it may be productive to look beyond the $\Ext^1$ case and explore whether modular Harish-Chandra theory yields useful information on the dimension of higher $\Ext$ groups in non-defining characteristic. \\

In another direction, we plan to continue to develop the computations of Section 8.3 involving the decomposition matrices for $\GL_n(q)$. These computations yield significantly improved bounds on the dimension of $\Ext^1$ in the cases considered. It may soon be possible to extend these computations to other finite groups of Lie type using recent work of Du, Parshall, and Scott \cite{dpsnew}, in which they construct an analog of the $q$-Schur algebra outside of Type A. Additionally, we may be able to use certain ideas involved in the computations of Section 8.3 to obtain new bounds on the dimension of $\Ext^1$ (and perhaps higher $\Ext$ groups) between irreducible modules for finite groups of Lie type.

\section{Acknowledgements}

I would like to thank my PhD advisor, Brian Parshall, for countless useful discussions and advice throughout the process of researching and writing this paper. I would  like to thank Leonard Scott for his feedback at various stages of the research and writing process. I would also like to thank Gunter Malle for his feedback on this paper and for his suggestion to use the decomposition matrices found in \cite{james} to obtain sharper bounds on the dimension of $\Ext^1$ for small general linear groups.


\newpage




\end{document}